\documentclass[oneside, 12pt]{amsart} \topmargin      -10mm \textwidth      160 true mm
\usepackage{amsmath, amssymb, amsfonts, amstext, amsthm, amscd}
\usepackage[mathscr]{euscript}
\usepackage{enumerate,url,hyperref}
\usepackage{tikz,color,soul,cleveref}
\usepackage[english]{babel}
\usepackage{chngcntr}
\usepackage{stmaryrd}
\usetikzlibrary{arrows.meta, positioning,arrows}
\newtheorem{theorem}{Theorem}[section]
\newtheorem{ques}[theorem]{Question}

\newtheorem{lem}[theorem]{Lemma}

\newtheorem{cor}[theorem]{Corollary}
\theoremstyle{definition}

\newtheorem{theorem'}{Theorem}[subsection]
\newtheorem{ques'}[theorem']{Question}
\newtheorem{prop'}[theorem']{Proposition}
\newtheorem{lem'}[theorem']{Lemma}
\newtheorem{CI'}[theorem']{Combinatorial Identification}
\newtheorem{cor'}[theorem']{Corollary}
\newtheorem{definition'}[theorem']{Definition}
\newtheorem{eg'}[theorem']{Example}
\newtheorem{remark'}[theorem']{Remark}

\newtheorem{theorem''}{Theorem}[subsubsection]
\newtheorem{ques''}[theorem'']{Question}
\newtheorem{prop''}[theorem'']{Proposition}
\newtheorem{lem''}[theorem'']{Lemma}
\newtheorem{CI''}[theorem'']{Combinatorial Identification}
\newtheorem{cor''}[theorem'']{Corollary}
\newtheorem{definition''}[theorem'']{Definition}
\newtheorem{eg''}[theorem'']{Example}
\newtheorem{remark''}[theorem']{Remark}

\numberwithin{equation}{section}
\title[Licci Complementary Edge Ideals]{Licci Property of Complementary Edge Ideals}
 \address{Department of Mathematics, Indian Institute of Technology Kharagpur, Kharagpur, India - 721302.}
\email{v.bhabani.lama@gmail.com, vivekbhabanilama@kgpian.iitkgp.ac.in}
\author{Vivek Bhabani Lama}

\date{}

\begin{document}
	
	\begin{abstract}
		In this article, we characterize the class of complementary edge ideals which satisfy the licci property in terms of the underlying graph. Using this characterization, we associate the licci property of a complementary edge ideal to its other algebraic properties. Finally, we provide two different probability regimes for which the complementary edge ideals of random graphs are licci with high  probability and not licci with high probability respectively. 
	\end{abstract}
	
	\maketitle
  
\section{Introduction}
 Linkage theory is a classical and well studied aspect of algebraic geometry. Peskine and Szpiro in \cite{peskine1974liaison} converted the general theory of linkage into its ideal theoretic counter-part, which is an active area of research in commutative algebra for over many decades (see for instance \cite{dibaei2015linkage}, \cite{MR4658822}, \cite{huneke1987structure},   \cite{Liason} and \cite{klein2021geometric} to mention a few). One aspect of keen interest in studying linkage theory from the ideal theoretic prospective of commutative algebra has been to study the translation of the properties between two linked ideals. In this regard, researchers have extensively studied the class of \textit{licci} ideals i.e., the class of ideals that fall into the linkage class of a complete intersection (see for instance \cite{RIareLICCI} and \cite{HL}) .    

In combinatorial commutative algebra, the main aim in this regard has been to study the licci property of the ideal in a polynomial ring via the combinatorics of the associated combinatorial object (see for instnace \cite{4D} and \cite{LLSR}). For the class of edge ideals a complete characterization of licci edge ideals was given in \cite[Theorem~3.7.]{LicciEdgeIdeals}. A similar result was formulated for the class of binomial edge ideals in \cite[Theorem~3.5.]{LicciBinom}. Very recently, a new class of ideals associated to graphs was introduced in \cite{hibi2025complementary} and independently in \cite{ficarra2025complementary}, under the nomenclature of complementary edge ideals (for definitions and details see \Cref{S2.1}).  Fundamental algebraic properties such as the characterization of the class of these ideals that admit a linear resolution (see \cite[Theorem~4.5.]{ficarra2025complementary} and \cite[Theorem~2.2.]{hibi2025complementary}), primary decomposition of these ideals (see \cite[Theorem~2.1.]{ficarra2025complementary}), expressions for regularity and projective dimensions (see for instanace \cite[Theorem~4.1.]{ficarra2025complementary} and \cite[Theorem~4.2.]{hibi2025complementary}) and various other algebraic properties were studied in  \cite{ficarra2025complementary} and \cite{hibi2025complementary}.

The main goal of this paper is to characterize the complementary edge ideals that satisfy the licci property in terms of the associated graphs. Thus, we prove the following theorem: 
\vskip 2mm
\noindent
	{\bf Theorem A \bf (Theorem \ref{main theorem}).}
 \textit{  Let $G$ be a simple graph on $\{1,2,\dots,n\}$. Then, the following conditions are equivalent:
 \begin{enumerate}
\item  $(I_c(G))_\mathfrak{m} \subseteq R$ is licci,
\item $G$ is a forest or $G=K_3$, the complete graph on $3$ vertices.
\end{enumerate}
}
\vspace{2mm}
We observe that this combinatorial classification provides an association to other algebraic properties of complementary edge ideals which seem unrelated apriori to the result in \Cref{main theorem} (see \Cref{remark about other properties}). Another application of \Cref{main theorem} is the characterization of licci property of complementary edge ideals via their Castelnuovo-Mumford regularity and projective dimension (see \Cref{cor1} and \Cref{cor 2}). 

In recent years, the use of probabilistic techniques in combinatorial commutative algebra has become an area of active research (see for instance \cite{de2019random}, \cite{dochtermann2023random}, \cite{erman2018random},  \cite{Newman} and the references therein). Motivated by the use of Erd\H{o}s-Rényi random graphs in \cite{banerjee2023edge} to determine a random version of asymptotic behaviour of invariants of edge ideals, we give another application of \Cref{main theorem} in order to provide two different probability regimes for which the complementary edge ideals of random graphs are licci with high  probability and not licci with high probability respectively (see \Cref{prob corollary}).
\vspace{1mm}

The paper is organized into three sections. In \Cref{S2}, we give the basic definitions and introduce some notations that will be used throughout the article. Some properties of complementary edge ideals of graphs studied in the literature are mentioned in this section. In \Cref{S3}, we prove \Cref{main theorem}, the characterization of licci complementary edge ideals in terms of their underlying simple graph. We also give some applications of our main result in this section. 

\section{Preliminaries}\label{S2}
In this section, we introduce various notations and preliminaries used throught the article.
Throughout, let $\mathbb{K}$ denote a field and $S=\mathbb{K}[x_{1},x_{2},\dots,x_{n}]$ be the polynomial ring with indeterminates $x_{1},x_{2},\dots x_{n}$ over the field $\mathbb{K}$  where $n \in \mathbb{N}$. Let $\mathfrak{m}$ denote the homogeneous maximal ideal $(x_1,x_2,\dots,x_n)$ and let $R=S_\mathfrak{m}$, i.e. the localization of $S$ at the homogeneous maximal ideal $\mathfrak{m}$.  Let $ M $ be a finitely generated graded $ S $-module. Consider the graded minimal free resolution of $ M $: 
$$ 0 \to S_p=\bigoplus_{j\in\mathbb{Z}}S(-j)^{\beta_{p, j}(M)}\xrightarrow{\phi_p}\dots \xrightarrow{\phi_1}S_0=\bigoplus_{j\in\mathbb{Z}}S(-j)^{\beta_{0, j}(M)} \xrightarrow{\phi_0} M \to {0}, $$
where $ S(-j) $ is the twist of $ S $ by degree $ j $, i.e., $ S(-j)_n:=S_{n-j} $ for all $ n\in\mathbb{Z} $. Using the above construction of minimal free resolution of $M$, we have the following invariants:
$$ \mathrm{reg}(M):=\max\{j-i:\beta_{i,j}(M) \neq 0\},$$
$$ \mathrm{pd}(M):=\max\{i:S_i\neq 0\}. $$

\noindent Finally, for a graph $G=(V(G),E(G))$, let $V(G)$ denote the vertex set of the graph $G$ and $E(G)$ denote the edge set of $G$. 
\subsection{Complementary Edge Ideals}\label{S2.1} In this subsection, we recall the definitions and results about complementary edge ideals in the literature. The class of \textit{Complementary Edge Ideals} was first introduced in \cite{hibi2025complementary} and independently in \cite{ficarra2025complementary}. 
\begin{definition'}
    Let $G$ be a simple graph with $V(G)=\{1,2,\dots,n\}$ and edge set $E(G)$. The \textit{complementary edge ideal} associated to $G$ is the ideal in the polynomial ring $S=\mathbb{K}[x_{1},x_{2},\dots,x_{n}]$ given by:
    $$I_c(G)= (x_1x_2...x_n\setminus x_ix_j\ |\  \{i,j\} \in E(G) ).$$
\end{definition'}
\noindent Fundamental properties of these ideals are well-studied in \cite{ficarra2025complementary} and \cite{hibi2025complementary}. We state a few of these results for completeness. The following result characterizes the Cohen-Macaulayness of complementary edge ideals in terms of its underlying simple graph. 
\begin{theorem'}\cite[Theorem~2.8.]{ficarra2025complementary}, \cite[Theorem~1.4.]{hibi2025complementary} \label{CM} Let $G$ be a simple graph then the following are equivalent:
\begin{enumerate}
    \item $S/I_c(G)$ is Cohen-Macaulay.
    \item $G$ is either a complete graph or a forest.
\end{enumerate}
\end{theorem'}

\noindent The following result provides bounds on the projective dimension and Castelnuovo-Mumford regularity of complementary edge ideals. 
\begin{theorem'}\cite[Proposition~4.1.]{hibi2025complementary}\label{bound on reg and pd}
Let $G$ be a simple graph on the vertex set $\{1,2,\dots,n\}$ and $I_c(G)$ be the complementary edge ideal of $G$, which is an ideal in $S=\mathbb{K}[x_1,x_2,\dots,x_n]$. Then, the following conditions hold:
\begin{enumerate}
    \item $1 \leq \mathrm{pd}(I) \leq 2$ , where $\mathrm{pd}(I)$ represents the projective dimension of $I$. 
    \item $n-2\leq \mathrm{reg}(I) \leq n-1$
\end{enumerate}
\end{theorem'}
The following theorem provides a characterization of Cohen-Macaulay height $2$ complementary edge ideals in terms of their projective dimension. 
\begin{theorem'}\cite[Theorem~1.1.]{hibi2025complementary} \label{pd lem}
Let $G$ be a simple graph on the vertex set $\{1,2,\dots,n\}$, which is not the complete graph $K_n$ and $I_c(G)$ be the complementary edge ideal of $G$, which is an ideal in $S=\mathbb{K}[x_1,x_2,\dots,x_n]$. Then $\mathrm{ht}(I_c(G))=2$ and $S/I_c(G)$ is Cohen-Macaulay if and only if $\mathrm{pd}(I_c(G))=1$.
\end{theorem'}
\noindent In particular, we mention the explicit regularity of the complementary edge ideal of a forest. 
\begin{theorem'}\cite[Theorem~4.2.]{hibi2025complementary}\label{char using reg and pd}
Let $G$ be a simple graph on the vertex set $\{1,2,\dots,n\}$, and $I_c(G)$ be the complementary edge ideal of $G$, which is an ideal in $S=\mathbb{K}[x_1,x_2,\dots,x_n]$. Then, the following hold:
\begin{enumerate}
    \item $\mathrm{pd}(I)=1$ and $\mathrm{reg}(I)=n-2$ if and only if $G$ is a tree or a complete graph.
    \item  $\mathrm{pd}(I)=1$ and $\mathrm{reg}(I)=n-1$ if and only if $G$ is a disconnected forest.
\end{enumerate}
\end{theorem'}
\subsection{Erd\H{o}s-Rényi Random Graphs}
Let $ G $ be a graph and $ G(n, p) $ denote the Erd\H{o}s-Rényi random graph where $ n\geq 1 $ and $ p\in [0, 1] $. The vertex set of $ G(n, p) $ is $ \{1,2,\dots , n\} $ and each set $ \{i, j\}\subseteq \{1,2,\dots,n\} $ with $ i\neq j $ is an edge in $ G(n, p) $ with probability $ p $ and is independent of other edges. More precisely, given a graph $ G $ with $V(G)= \{1,2,\dots,n\}$ and $|E(G)|= m $,
$$ \mathbb{P}\{G(n, p)=G\}=p^m(1-p)^{\begin{pmatrix}n\cr 2\end{pmatrix}-m}. $$
The above expression shows that every graph with $ m $ edges has an equal chance of being chosen. When $ p>1 $, $ G(n, p) $ is defined as $ G(n, \min(p, 1)) $. Typically, one examines the properties of $ G(n, p) $ as $ n \to \infty $ and $ p $ also changes with $ n $. It is standard to write $ p $ even though it is understood that $ p:=p(n) $. By the properties of $ G(n, p) $, we refer to combinatorial properties, which are the characteristics of graphs that remain unchanged under graph isomorphisms.  We will denote the complementary edge ideal  of $ G(n, p)$ as the random ideal instead of $I_c(G(n,p))$.
The following theorem provides probability regimes for which a  certain graph appears as a subgraph of an Erd\H{o}s-Rényi random graph:
\begin{theorem'}\cite[Theorem~5.3.]{frieze2015introduction}\label{prob}
  Let $H$ be a fixed graph with $|E(H)|>0$. Then, we have: $$\displaystyle\lim_{{n \rightarrow \infty}}\mathbb{P}[H \subseteq G(n,p)]= \begin{cases} 
      1 & n^{(\frac{1}{m(H)})}p \rightarrow 0 \\
      0 & n^{(\frac{1}{m(H)})}p  \rightarrow \infty 
   \end{cases}
$$  where $m(H)=\max\left\{d(K)=\frac{|E(K)|}{|V(K)|} \ | \ K \text{ is a subgraph of $H$} \right\}$. 
\end{theorem'}

\subsection{Licci Ideals} In this subsection, we recall the basic definitions and a few results about licci ideals. As the nomenclature suggests, licci ideals are ideals are the class of ideals that are in the linkage class of a complete intersection ideal. 
\begin{definition'}
   Let  $R$ be a regular local ring and let $I$ and $J$ be two proper ideals of $R$. $I$ and $J$ are called \textit{directly linked} if there exist a regular sequence $\boldsymbol{\alpha}=\alpha_1,\dots,\alpha_n $ in $I \cap J$ such that $I=(\boldsymbol{\alpha}):J$ and $J=(\boldsymbol{\alpha}):I$.  
\end{definition'}
\begin{remark'}
Two directly linked ideals $I$ and $J$ will be denoted as $I \sim J$.
\end{remark'}   
\noindent Next we recall the notion of a linkage class of an ideal in a regular local ring.

\begin{definition'}\label{Linkage class def}
Two ideals $I$ and $J$ are said to be in the same \textit{linkage class} if there exists a sequence of direct links of the follwoing form: 
\begin{center}
    $I=I_0\sim I_1 \sim \dots \sim I_m = J $
\end{center}
\end{definition'}
\begin{remark'}
 In \Cref{Linkage class def}, if $J$ is a complete intersection ideal, then $I$ is said to be in the linkage class of a complete intersection (often reffered to as $I$ is a licci ideal)  
\end{remark'}
 
  A necessary condition for a homogeneous ideal in a polynomial ring to be licci is given in \cite{huneke1987structure}. This is an essential tool in order to rule out various classes of non-licci ideals. Therefore, we recall the following theorem from the literature: 

\begin{theorem'}\cite[Corollary~5.12.]{huneke1987structure}\label{reg condition}
     Let $I$ be a Cohen-Macaulay homogeneous ideal in a standard graded polynomial ring $S = \mathbb{K}[x_1,x_2,\dots,x_n]$ with the graded maximal
 ideal $\mathfrak{m}=(x_1,x_2,\dots,x_n)$. If $I_{\mathfrak{m}} \subset R = S_{\mathfrak{m}}$ is licci, then
 \begin{center}
     $\mathrm{reg} (S/I) \geq (\mathrm{ht}(I)-1)(\mathrm{indeg}(I)-1)$
 \end{center}
where $\mathrm{indeg}(I)=\mathrm{min} \{i : I_i \neq 0\}$ represents the initial degree of the ideal $I$.

\end{theorem'}
\section{Complementary Edge Ideals and the Licci Property}\label{S3}
    In this section, we prove our main result which is the characterization of licci complementary edge ideals. We prove that the class of complementary edge ideals that satisfy this property are the ones whose underlying graphs are forests. To prove our main theorem, we first prove a result on the height of the complementary edge ideal of a complete graph. While this result is previously established in (\cite[Corollary~2.3.]{ficarra2025complementary}), our proof employs a different approach via its connections to edge ideals of complete hypergraphs and we include it for the sake of completeness.  

\begin{lem}\label{lemma on height}
Let $n\geq 3$ be a positive integer and $G=K_n$ be a complete graph on the vertex set $\{1,2,\dots,n\}$. Then, $\mathrm{ht}(I_c(G))=3$.   
\end{lem}
\begin{proof}
 We observe that for a complete graph $G=K_n$, the ideal $I_c(G)$ is an ideal generated by all possible squarefree monomials of degree $n-2$. Also, for a monomial ideal $J$ generated by all possible squarefree monomials of a fixed degree $d$, it is well-known that the height of $J$ is given by $n-d+1$. Therefore, it easily follows that $\mathrm{ht}(I_c(G))=n-(n-2)+1=3$.     
\end{proof}

Now, we prove our main result which characterizes the licci complementary edge ideals.
\begin{theorem}\label{main theorem}
 Let $G$ be a simple graph on $\{1,2,\dots,n\}$. Then, the following conditions are equivalent:
 \begin{enumerate}
\item  $(I_c(G))_\mathfrak{m} \subseteq R$ is licci,
\item $G$ is a forest or $G=K_3$, the complete graph on $3$ vertices.
\end{enumerate}
\end{theorem}
\begin{proof}
For (1) implies (2), let $(I_c(G))_\mathfrak{m}$ be licci then, $S/I_c(G)$ is Cohen-Macaulay. Therefore, by \Cref{CM}, $G$ is either a complete graph or $G$ is a forest. Suppose that $G$ is a complete graph, hence $G=K_n$ then, by \Cref{lemma on height}, we have $\mathrm{ht}(I_c(K_n))=3$. Therefore, by \Cref{reg condition}, we have the following condition on the regularity of $I_c(G)$:
$$\mathrm{reg}(I_c(G))\geq (\mathrm{ht}(I_c(G))-1)(\mathrm{indeg}(I_c(G)-1)+1.$$
However, the $\mathrm{indeg}(I_c(G))=n-2$ and hence we arrive at the following inequality for regularity:
$$\mathrm{reg}(I_c(G)) \geq(\mathrm{ht}(I_c(G)-1)((n-2)-1)+1=(3-1)(n-3)+1.$$
Therefore, we get that the regularity is given by $\mathrm{reg}(I_c(G))\geq 2n-6+1=2n-5$. However, from \Cref{bound on reg and pd}, we have that $\mathrm{reg}(I_c(G)) \leq n-1$. Clearly for $n \geq 5$, we have  $2n-5 > n-1$ which gives a contradiction and hence $G$ is not $K_n$ for $n \geq 5$. Therefore, $G$ is a either a forest or $G=K_n$ with $n=3,4$. For $n=4$ the complementary edge ideal of $K_4$ coincides with the edge ideal of $K_4$, that is $I_c(K_4)=I(K_4)$. However, $(I(K_4))_{\mathfrak{m}}$ is  not  licci in $R_\mathfrak{m}$ (see \cite[Theorem~3.7.]{LicciEdgeIdeals})  For (2) implies (1): Suppose $G$ is a forest, then by \Cref{char using reg and pd}, the projective dimension $\mathrm{pd}(I_c(G
))=1$ and since $G$ is not complete, therefore by \Cref{pd lem}, $S/I_c(G)$ is Cohen-Macaulay and $\mathrm{ht}(I_c(G))=2$. Since $I_c(G)$ is a height 2 ideal such that  $S/I_c(G)$ is Cohen-Macaulay. Therefore,  $(I_c(G))_\mathfrak{m} \subseteq R$ is a licci ideal (see \cite{peskine1974liaison} for details). Next if $G=K_3 \text{ or } K_4$, for $n=3$ we have $G=K_3$ and hence $I_c(G)=I_c(K_3)=(x_1,x_2,x_3)$ which is a complete intersection ideal and hence licci. This completes the proof.
\end{proof}
A consequence of \Cref{main theorem}, is the following characterization of licci complementary edge ideals whose graphs have more than one connected components in terms of their regularity and projective dimension.
\begin{cor}\label{cor1}
Let $G$ be a simple graph on the vertex set $\{1,2,\dots,n\}$ with more than one connected component. Then, the following are equivalent:
\begin{enumerate}
    \item  $(I_c(G))_\mathfrak{m} \subseteq R$ is licci,
    \item $S/I_c(G)$ is Cohen-Macaulay,
\item $G$ is a forest,
\item $\mathrm{reg}(I_c(G))=n-1$ and $\mathrm{pd}(I)=1$.
\end{enumerate}
\end{cor}
\begin{proof}
The proof is clear due to \Cref{main theorem}, \Cref{CM} and \Cref{char using reg and pd}.     
\end{proof}

Using this characterization, we get the following corollary regarding the regularity and projective dimension of complementary edge ideals of connected graphs that have the licci property. 
\begin{cor}\label{cor 2}
Let $G$ be a simple connected graph on the vertex set $\{1,2,\dots,n\}$ which is not $K_n$, the complete graph on $n$ vertices. Then, the following are equivalent:
\begin{enumerate}
    \item  $(I_c(G))_\mathfrak{m} \subseteq R$ is licci,
\item $G$ is a tree,
\item $S/I_c(G)$ is Cohen-Macaulay,
\item $\mathrm{reg}(I_c(G))=n-2$ and $\mathrm{pd}(I)=1$.
\end{enumerate}
\end{cor}
\begin{proof}
The proof is clear due to \Cref{main theorem}, \Cref{CM} and \Cref{char using reg and pd}. 
\end{proof}
\noindent As a consequence of \Cref{main theorem}, we have the following result which gives the explicit probability that an Erd\H{o}s-Rényi random graph will correspond to a licci complementary edge ideal.

\begin{cor}\label{prob corollary}
 Let $G(n,p)$ be an Erd\H{o}s-Rényi random graph. Then, the following hold:
   $$\displaystyle\lim_{{n \rightarrow \infty}}\mathbb{P}[I_c(n,p) \text{ is not licci}]= \begin{cases} 
      1 & np \rightarrow 0 \\
      0 & np \rightarrow \infty 
   \end{cases}
$$ 
\end{cor}
\begin{proof}
From \Cref{main theorem}, it is enough to show that the expression for the limit of the probability of $G(n,p)$ containing some cycle as a subgraph $H=C_m$ for some $m \geq 4$, is given by:
$$\displaystyle \lim_{{n \rightarrow \infty}}\mathbb{P}[H \text{ is a subgraph of }G(n,p) ]= \begin{cases} 
      1 & np \rightarrow 0 \\
      0 & np \rightarrow \infty 
   \end{cases}$$
Using \Cref{prob}, it is enough to show the following equality:
$$m(H)=\max\left\{d(K)=\frac{|E(K)|}{|V(K)|} \ | \ K \text{ is a subgraph of $H$} \right\}=1.$$
However since $H=C_m$ is a cycle for some $m > 3$, the only possibilities for $K$ are either $K=H=C_m$ or $K$ is a path with $t_1$ many vertices where $t_1 \leq m$ or $K$ is a disjoint union    of paths with $t_2$ many vertices where $t_2 < m$. In the first case, if $K=C_m$ it is clear that $d(K)=1$ as the number of vertices of a cycle is equal to the number of edges. In the case that $K$ is a path with $t_1$ many edges, $d(K)=\frac{t_1-1}{t_1}=1-\frac{1}{t_1}<1$. Finally, in the case when $K$ is a disjoint union of paths with $t_2$ many edges, $d(K)<1$ as the number of vertices is more than the number of edges in this case. Therefore, $m(H)=1$ and hence the proof.    
\end{proof}

Since \Cref{main theorem} gives a combinatorial characterization of complementary edge ideals that are licci, we may now associate the licci property to various seemingly unrelated algebraic properties of complementary edge ideals (unrelated in the sense that apriori to the combinatorial characterization in \Cref{main theorem}, there is no reason to believe that one of these algebraic properties would imply the other) . Given that the properties of complementary edge ideals have been extensively studied in \cite{ficarra2025complementary} and \cite{hibi2025complementary}, we have the following non-reversible implications from \Cref{main theorem}:
\begin{cor}\label{remark about other properties}

\noindent If $(I_c(G))_{\mathfrak{m}}$ is licci then the following hold:
\begin{enumerate}
    \item $I_c(G)$ is sequentially Cohen-Macaulay. 
    \item  $I_c(G)^{\vee}$ is component wise linear.
    \item  $I_c(G)^{\vee}$ has linear quotients.
    \item $I_c(G)^{\vee}$ has linear resolution.
    \item  $I_c(G)$ has linear resolution.
\end{enumerate}
Furthermore, none of these implications are reversible. 
\end{cor}
\begin{proof}
    The corollary is a straightforward application of the combinatorial characterization in \Cref{main theorem} and various properties proved in \cite[Theorem~2.5.]{ficarra2025complementary} and \cite[Theorem~2.2.]{hibi2025complementary}.
\end{proof}
\noindent Another very similar direction in which research in commutative algebra is actively conducted is the classification and study of glicci ideals (see for instance \cite{migliore2002monomial}, \cite{migliore2013glicci} and \cite{nagel2008glicci}). In view of this, we conclude this section by posing the following question:
\begin{ques}
For what class of simple graphs $G$ does the complementary edge ideal $I_c(G)$ fall into the gorenstein linkage class of a complete intersection ideal? 
\end{ques}
\section*{Acknowledgments}
The author would like to thank his supervisor Dr. Arindam Banerjee for his valuable suggestions and guidance throught the project. 
The author would also like to sincerely thank Dr. Kanoy Kumar Das for his valuable discussions and pointing out various subtle ideas and technicalities, Ms. Shruti Priya for reading the manuscript and pointing out grammatical errors. The author is thankful to the Government of India for supporting him in this work through the Prime Minister Research Fellowship. The author acknowledges the use of the computer algebra system Macaulay2 \cite{M2} for testing their computations.
\section*{Funding}
\noindent The author is supported by PMRF fellowship, India.
\section*{Conflict of interest}
\noindent The author declares that there are no conflicts of interest regarding the publication of this article.
%\end{acknowledgement}
%\end{acknowledgement}
\bibliographystyle{amsplain}	
\bibliography{refs}

\end{document}